\documentclass{lmsplain}

\usepackage{amsmath,amssymb,amsfonts,enumerate,tikz,url,booktabs}

\newcommand{\Q}{\mathbb{Q}}
\newcommand{\Z}{\mathbb{Z}}

\newcommand{\M}{\ensuremath{\mathsf{M}}}

\newcommand{\tr}{\operatorname{tr}}
\newcommand{\Frob}{\operatorname{Frob}}
\newcommand{\norm}[1]{\Vert #1 \Vert}

\newtheorem{theorem}{Theorem}[section] 

\newnumbered{assertion}{Assertion}    
\newnumbered{conjecture}{Conjecture}  
\newnumbered{definition}{Definition}
\newnumbered{hypothesis}{Hypothesis}
\newnumbered{remark}{Remark}
\newnumbered{note}{Note}
\newnumbered{observation}{Observation}
\newnumbered{problem}{Problem}
\newnumbered{question}{Question}
\newnumbered{algorithm}{Algorithm}
\newnumbered{example}{Example}
\newunnumbered{notation}{Notation} 

\title[Computing Hasse--Witt matrices]
 {\vspace{-50pt}Computing Hasse--Witt matrices of hyperelliptic curves\\in average polynomial time} 

\author{David Harvey and Andrew V. Sutherland}

\classno{11G20 (primary) 11Y16, 11M38, 14G10 (secondary)}

\extraline{The first author was supported by the Australian Research Council, DECRA Grant DE120101293.
The second author was supported by NSF grant DMS-1115455.}
 
\begin{document}
\maketitle

\begin{abstract}
We present an efficient algorithm to compute the Hasse--Witt matrix of a hyperelliptic curve $C/\Q$ modulo all primes of good reduction up to a given bound $N$, based on the average polynomial-time algorithm recently introduced by Harvey. An implementation for hyperelliptic curves of genus 2 and 3 is more than an order of magnitude faster than alternative methods for $N = 2^{26}$.
\end{abstract}


\vspace{-24pt}

\section{Introduction}
Let $C/\Q$ be a smooth projective hyperelliptic curve of genus $g$ defined by an affine equation
\[
 y^2 = f(x) = \sum_{i=0}^{d} f_i x^i, \qquad f_i \in \Z,
\]
where $d=\deg f$ is either $2g+1$ or $2g+2$ (generically, $d=2g+2$).
If $C$ has good reduction at an odd prime $p$, the associated \emph{Hasse--Witt matrix} $W_p=[w_{ij}]$ is the $g\times g$ matrix over $\Z/p\Z$ with entries
\[
w_{ij} = f^{(p-1)/2}_{pi-j}\bmod p\qquad (1\le i,j\le g),
\]
where $f^n_k$ denotes the coefficient of $x^k$ in $f(x)^n$; see \cite{Gonzalez:HasseWitt,Yui:HasseWittMatrix}.
We have the identity 
\begin{equation}\label{eq:charpoly}
\chi(\lambda) \equiv (-1)^g\lambda^g\det(W_p - \lambda I)\bmod p,
\end{equation}
where $\chi(\lambda) \in \Z[\lambda]$ is the characteristic polynomial of the Frobenius endomorphism of the Jacobian of the reduction of $C$ at $p$; see  \cite{Manin:HasseWittMatrix}.
In particular, the Weil bounds imply that for $p>16g^2$ the trace of $W_p$ uniquely determines the trace of Frobenius, hence the number of points $p+1-\tr(\Frob_p)$ on the reduction of $C$ at $p$.

We say that a prime $p$ is \emph{admissible (for $C$)} if $p$ is odd, $C$ has good reduction at $p$, and $p$ does not divide $f_0$ or $f_d$ (the constant and leading coefficients of $f$). The goals of this paper are to give a fast algorithm for computing $W_p$ simultaneously for all admissible primes $p$ up to a given bound $N$, and to demonstrate the practicality of the algorithm for $g = 2$ and $g = 3$.
Applications include numerical investigations of the generalized  Sato--Tate conjecture \cite{FKRS:SatoTate,KS:SatoTate} and computing the $L$-series of $C$ \cite{KS:Lseries}.

The algorithm presented here is inspired by \cite{Harvey:HyperellipticPolytime}, which gives an algorithm to compute~$\chi(\lambda)$ (not just $\chi(\lambda) \bmod p$) for all primes $p\le N$ of good reduction, in the case that $d$ is odd (which implies that $C$ has a rational Weierstrass point).
The running time of that algorithm is $O(g^{8+\epsilon} N\log^{3+\epsilon} N)$; when averaged over primes $p\le N$, this is $O(g^{8+\epsilon} \log^{4+\epsilon} p)$, the first such result that is polynomial in both $g$ and $\log p$.
Critically, the exponent $4$ of $\log p$ does not depend on $g$, and it is already better than that of Schoof's algorithm \cite{Schoof:Polytime} in genus 1, which has an exponent of~$5$ when suitably implemented.\footnote{This assumes fast integer arithmetic is used, which we do throughout.  Under heuristic assumptions, the (probabilistic) SEA algorithm reduces the exponent to $4$, but for $g=1$ generic algorithms that run in $O(p^{1/4+\epsilon})$ time are superior within the feasible range of $p\le N$ in any case.}
Pila's generalization of Schoof's algorithm~\cite{Pila:Polytime} has an exponent of $8$ in genus $2$ (see \cite{GKS:g2Pila,GaudrySchost:g2Pila}), and Eric Schost has suggested (personal communication) that the exponent is $12$ in genus~$3$ (Pila's bound in \cite{Pila:Polytime} gives a much larger exponent).

For our implementation we focus on the cases $g \le 3$, where knowledge of $\chi(\lambda)\bmod p$ allows one to efficiently determine $\chi(\lambda)$ using a generic group algorithm, as described in \cite{KS:Lseries}. When $g = 3$, the time required to deduce $\chi(\lambda)$ from $\chi(\lambda) \bmod p$ is $O(p^{1/4+\epsilon})$; while this is exponential in $\log p$, for $p\le N$ it is actually negligible compared to the $O(\log^{4+\epsilon} p)$ average time to compute $\chi(\lambda) \bmod p$ using the method of this paper, within the feasible range of $N$ (say $N\le 2^{32}$).
We handle all hyperelliptic curves, not just those with a rational Weierstrass point, which in general will not be present.
We also introduce optimizations that improve the space complexity by a logarithmic factor, compared to \cite{Harvey:HyperellipticPolytime}, without increasing the running time; indeed, the running time is significantly reduced, as may be seen in Table~\ref{table:HWtimingsk} in \S\ref{sec:details}.

Asymptotically, we obtain the following theorem bounding the complexity of the algorithm \textsc{ComputeHasseWittMatrices}, which computes $W_p$ for all admissible $p \le N$ (see \S \ref{sec:ART} for the algorithm and a proof of the theorem). We denote by $\norm f$ the maximum of the absolute value of the coefficients of $f$, and by $\M(n)$ the time to multiply two $n$-bit integers. We may take $\M(n) = O(n\log n\log\log n)$, via \cite{SS:IntegerMultiplication}.

\vspace{-8pt} 

\begin{theorem}\label{thm:main}
Assume that $g=O(\log N)$. The running time of the algorithm \textsc{ComputeHasseWittMatrices} is
\[
 O(g^5 \M(N\log (\norm f N))\log N),
\]
and it uses
\[
 O\left(g^2 N\left(1+ \frac{\log \norm f}{\log N}\right)\right) 
\]
space.
\end{theorem}

\vspace{-8pt} 

Assuming $\log \norm f$ grows no faster than $\log N$, the bounds in Theorem~\ref{thm:main} simplify to $O(g^5N\log^{3+\epsilon}N)$ time and $O(g^2N)$ space.

In practical terms, the new algorithm is substantially faster than previous methods.
We benchmarked our implementation against two of the fastest software packages available for these computations, as analyzed in \cite{KS:Lseries}: the \texttt{hypellfrob} \cite{hypellfrob} and  \texttt{smalljac} \cite{smalljac} software libraries. 
In genus 2 the new algorithm outperforms both libraries for $N\ge 2^{19}$, and is more than 10 times faster for $N=2^{26}$.
In genus 3 the new algorithm is faster across the board, and more than 20 times faster for $N=2^{26}$.
Key to achieving these performance improvements are a faster and more space-efficient algorithm for computing the accumulating remainder trees that play a crucial role in \cite{Harvey:HyperellipticPolytime}, and an optimized FFT implementation for multiplying integer matrices with very large coefficients.
\vspace{-16pt}

\section{Overview}\label{sec:OV}

Each row of the Hasse--Witt matrix $W_p$ of $C$ consists of of $g$ consecutive coefficients of $f^n$ reduced modulo $p$, where $n=(p-1)/2$.
The total size of all the polynomials~$f^n$ needed to compute $W_p$ for $p \leq N$ is $O(N^3\norm f)$ bits; this makes a na\"ive approach hopelessly inefficient. Two key optimizations are required to achieve a running time that is quasilinear in~$N$.

First, for a given row of $W_p$, we only require $g$ coefficients of each $f^n$.  In \S\ref{sec:RR} we define an $r$-dimensional row vector $v_n$, where $r \approx 2g$, consisting of $r$ consecutive coefficients of $f^n$, including the $g$ coefficients of interest. The coefficients of $f^{n+1}$ corresponding to $v_{n+1}$ are closely related to the coefficients of $f^n$ corresponding to $v_n$.
We use this to derive a linear recurrence $v_{n+1} = v_n T_n$, where $T_n$ is an explicit $r \times r$ transition matrix. The entries of $T_n$ lie in $\Q$, but not necessarily in $\Z$; this requires us to handle the denominators explicitly. These recurrence relations are analogous to the technique of ``reduction towards zero'' introduced in \cite{Harvey:HyperellipticPolytime}; the key point is that the coefficients of the recurrence are \emph{independent of $p$}. This is in contrast to the recurrence relations used to derive the Hasse--Witt matrix in \cite{BGS-recurrences}, whose coefficients do depend on $p$, and which are analogous to the ``horizontal reductions'' in \cite{Har-kedlaya} and \cite{Harvey:HyperellipticPolytime}.

Second, we only need to know the coefficients of each vector $v_n$ modulo $p=2n+1$.
The essential difficulty here is that the modulus is different for each $n$.
Following \cite{Harvey:HyperellipticPolytime}, we use an \emph{accumulating remainder tree} to circumvent this problem.
More precisely, in \S\ref{sec:ART} we give an algorithm \textsc{RemainderTree} that takes as input a sequence of integer matrices $A_0, \ldots, A_{b-2}$, a sequence of integer moduli $m_1, \ldots, m_{b-1}$, and an integer row vector $V$ (the ``initial condition''), and computes the reduced partial products (row vectors)
 \[ C_n := V A_0 \cdots A_{n-1} \bmod m_n, \]
simultaneously for all $0 \leq n < b$. The remarkable feature of this algorithm is that its complexity is quasilinear in $b$.

We may  apply \textsc{RemainderTree} to our situation in the following way. During the course of finding an explicit expression for $T_n$, we will write it as $T_n = M_n / D_n$ where $M_n$ is an integer matrix and $D_n$ is a nonzero integer. It turns out that for any sufficiently large admissible prime $p = 2n + 1$, the $p$-adic valuation of $D_0 \cdots D_{n-1}$ is at most $d$. Thus to obtain
 \[ v_n = v_0 M_0 \cdots M_{n-1} / D_0 \cdots D_{n-1} \]
modulo $p$, it suffices to compute
 \[ v_0 M_0 \cdots M_{n-1} \bmod p^{d+1}  \qquad\text{and}\qquad  D_0 \cdots D_{n-1} \bmod p^{d+1}. \]
We run \textsc{RemainderTree} twice, first with $V = v_0$ and $A_j = M_j$, and then with $V = 1$ and $A_j = D_j$ (regarding the $D_j$ as $1 \times 1$ matrices). In both cases we take the moduli $m_n = p^{d+1}$ if $p = 2n+1$ is an admissible prime, and let $m_n=1$ otherwise.

For $g \le 3$, we will show how to tweak this strategy to use the smaller moduli $m_n = p^g$.
This has a significant impact on the overall performance and memory consumption. 
We conjecture that one can always use $m_n=p^g$ (for $p$ sufficiently large compared to $g$), but we will not attempt to prove this here.

\vspace{-6pt}
\section{Recurrence relations}\label{sec:RR}

For technical reasons it will be convenient to distinguish between the cases $f_0 \neq 0$ and $f_0 = 0$ (the same distinction arises in \cite{Harvey:HyperellipticPolytime}). Let
 \[ r = \begin{cases} d & \text{if $f_0 \neq 0$}, \\ d-1 & \text{if $f_0 = 0$}. \end{cases} \]
For each $1 \leq i \leq g$, consider the sequence of vectors
 \[ v^{(i)}_n = [f^n_{2in+i-r}, \ldots, f^n_{2in+i-1}] \in \Z^r \qquad (n \geq 0). \]
For each admissible prime $p=2n+1$, the last $g$ entries of $v^{(i)}_n$ are, modulo $p$, precisely the entries of the $i$th row of the Hasse--Witt matrix $W_p$ (in reversed order).

The aim of this section is to develop a recurrence for the $v^{(i)}_n$.
For each $n \geq 0$, we will construct an $r \times r$ integer matrix $M^{(i)}_n$, and a nonzero integer $D^{(i)}_n$, such that
 \[ v^{(i)}_{n+1} = v^{(i)}_n M^{(i)}_n / D^{(i)}_n. \]
The entries of $M^{(i)}_n$, and $D^{(i)}_n$, turn out to be polynomials in $n$ and the coefficients of $f$,
which allows us to analyze the $p$-adic valuation of the partial products of the $D^{(i)}_n$.

The construction proceeds as follows. For any $n \geq 0$, the identities
\[
f^{n+1} = f f^n\qquad\text{and}\qquad(f^{n+1})' = (n+1)f'f^n
\]
imply the relations
\begin{align}\label{eq:rel1}
f^{n+1}_k &=\sum_{j=0}^d f_jf^n_{k-j},\\
kf^{n+1}_k &= (n+1)\sum_{j=1}^d jf_jf^n_{k-j}.\label{eq:rel2}
\end{align}
Multiplying \eqref{eq:rel1} by $k$ and subtracting \eqref{eq:rel2} yields the relation
\begin{equation}\label{eq:rel3}
\sum_{j=0}^d (nj - k + j)f_j f^n_{k-j} = 0
\end{equation}
among the coefficients of $f^n$.

Suppose we are in the case $f_0 \neq 0$, $r = d$. Solving \eqref{eq:rel3} for $f^n_k$ yields
\begin{equation}\label{eq:right}
k f_0 f^n_{k} = \sum_{j=1}^{d}(nj - k + j)f_j f^n_{k-j}.
\end{equation}
For $k \neq 0$, this expresses $f^n_k$ as a linear combination of $d$ consecutive coefficients of $f^n$ to the ``left'' of $f^n_k$.
On the other hand, replacing $k$ by $k + d$ and $j$ by $d - j$ in \eqref{eq:rel3} gives
\begin{equation}\label{eq:left}
(nd - k) f_d f^n_k = -\sum_{j=1}^d (n(d - j) - k - j) f_{d-j} f^n_{k+j}.
\end{equation}
For $k \neq nd$, this expresses $f^n_k$ as a linear combination of $d$ consecutive coefficients to the ``right'' of $f^n_k$. Now, suppose we are given as input
 \[ v^{(i)}_n = [f^n_{2in+i-d}, \ldots, f^n_{2in+i-1}]. \]
After $2i$ applications of \eqref{eq:right}, i.e., for $k = 2in + i, \ldots, 2in+3i-1$ (in that order), and $d - 2i$ applications of \eqref{eq:left}, i.e., for $k = 2in+i-d-1, \ldots, 2in+3i-2d$ (in that order), we have extended our knowledge of the coefficients of $f^n$ to the  vector 
 \[ [f^n_{2in+3i-2d}, \ldots, f^n_{2in+3i-1}]. \]
of length $2d$. From \eqref{eq:rel1} we then obtain
 \[ v^{(i)}_{n+1} = [f^{n+1}_{2in+3i-d}, \ldots, f^{n+1}_{2in+3i-1}]. \]

The above procedure defines a $d \times d$ transition matrix $T^{(i)}_n$ mapping $v_n^{(i)}$ to $v_{n+1}^{(i)}$, whose entries are rational functions in $\Q(n, f_0, \ldots, f_d)$. Denominators arise from the divisions by $kf_0$ and $(nd-k)f_d$ in the various applications of \eqref{eq:right} and \eqref{eq:left}. Each such divisor is a linear polynomial in $\Z[n]$ multiplied by either $f_0$ or $f_d$; thus the denominators of the entries of $T^{(i)}_n$ are polynomials in $\Z[n, f_0, f_d]$. We will take $D^{(i)}_n$ to be the least common denominator of the entries of $T^{(i)}_n$. Since there are $d$ applications of \eqref{eq:right} and \eqref{eq:left} altogether, the degree of $D^{(i)}_n$ with respect to $n$ is at most $d$ (it may be smaller due to cancellation).

The case $f_0 = 0$ with $r = d - 1$ is similar. We have $f_1 \neq 0$, because $f$ is assumed to be squarefree, and the analogues of \eqref{eq:right} and \eqref{eq:left} are
\begin{align}
       (n-k) f_1 f^n_k & = -\sum_{j=1}^{d-1} (n(j+1) - k + j) f_{j+1} f^n_{k-j}, \label{eq:right2} \\
    (nd - k) f_d f^n_k & = -\sum_{j=1}^{d-1} (n(d-j) - k - j) f_{d-j} f^n_{k+j}, \label{eq:left2}
\end{align}
which express $f^n_k$ in terms of $d - 1$ consecutive coefficients to the left, or right, of $f^n_k$. Given
 \[ v^{(i)}_n = [f^n_{2in+i-d+1}, \ldots, f^n_{2in+i-1}], \]
we use these relations to extend $v_n^{(i)}$ to the vector $[f^n_{2in+3i-2d+1}, \ldots, f^n_{2in+3i-1}]$ of length $2d-1$, from which we obtain $v^{(i)}_{n+1}$ from \eqref{eq:rel1} as above.

In the subsections that follow we carry out the above procedure explicitly for the specific cases that arise when $g\le 3$.

\subsection{Genus $1$, quartic model}
Suppose that $C/\Q$ has genus $1$.
If $C$ has a rational point, then $C$ is an elliptic curve and can be put in Weierstrass form $y^2=f(x)$ with $f$ cubic, but we first consider the generic case where this need not hold.
So let $f(x) = f_4 x^4 + f_3 x^3 + f_2 x^2 + f_1 x + f_0$ with $f_0f_4 \neq 0$; then $r = d = 4$.
Since the only relevant value of $i$ is $1$, we omit the superscripts on $v^{(1)}_n$, $M^{(1)}_n$, $D^{(1)}_n$.

We wish to construct a linear recurrence that expresses the vector
 \[ v_{n+1} = [f^{n+1}_{2n-1}, f^{n+1}_{2n}, f^{n+1}_{2n+1}, f^{n+1}_{2n+2}] \in \Z^4 \]
in terms of the vector
 \[ v_n = [f^n_{2n-3}, f^n_{2n-2}, f^n_{2n-1}, f^n_{2n}] \in \Z^4; \]
that is, we want a $4\times 4$ integer matrix $M_n$ and a nonzero integer $D_n$ such that
 \[ v_{n+1} = v_n M_n / D_n. \]
For each odd prime $p=2n+1$,  the Hasse--Witt matrix $W_p$ consists of just the single entry $f^n_{2n} \bmod p$, which is the last entry of $v_n\bmod p$.

We start by extending extending $v_n$ ``rightwards'', using \eqref{eq:right} with $k = 2n+1$.
This yields
 \[ (2n+1) f_0 f^n_{2n+1} = (2n+3)f_4 f^n_{2n-3} + (n+2)f_3 f^n_{2n-2} + f_2 f^n_{2n-1} - n f_1 f^n_{2n}. \]
Using \eqref{eq:right} again with $k = 2n+2$, we get
 \[ (2n + 2) f_0 f^n_{2n+2} = (2n+2) f_4 f^n_{2n-2} + (n+1) f_3 f^n_{2n-1} - (n+1) f_1 f^n_{2n+1}. \]
Combining these equations yields
\begin{align*}
 2(2n+1)f_0^2 f^n_{2n+2}
  &= -(2n+3)f_1 f_4 f^n_{2n-3} \\
  & \phantom{=}\  + \big(2(2n+1) f_0 f_4 - (n+2)f_1 f_3\big) f^n_{2n-2} \\
  & \phantom{=}\ + \big((2n+1) f_0 f_3 - f_1 f_2\big) f^n_{2n-1} \\
  & \phantom{=}\ + n f_1^2 f^n_{2n}.
\end{align*}

Next we extend $v_n$ ``leftwards'' by applying \eqref{eq:left} with $k = 2n-4$, obtaining
 \[ (2n+4) f_4 f^n_{2n-4} = -(n+3) f_3 f^n_{2n-3} - 2f_2 f^n_{2n-2} + (n-1) f_1 f^n_{2n-1} + 2n f_0 f^n_{2n}. \]
With $k = 2n - 5$ we get
 \[ (2n+5) f_4 f^n_{2n-5} = -(n+4) f_3 f^n_{2n-4} - 3f_2 f^n_{2n-3} + (n-2) f_1 f^n_{2n-2} + (2n-1) f_0 f^n_{2n-1}, \]
and therefore
\begin{align*}
(2n+5)(2n+4) f_4^2 f^n_{2n-5}
  & = \big( (n+3)(n+4) f_3^2 - 3(2n+4) f_2 f_4 \big) f^n_{2n-3} \\
  & \phantom{=}\ + \big( 2(n+4) f_2 f_3 + (n-2)(2n+4) f_1 f_4 \big) f^n_{2n-2} \\
  & \phantom{=}\ + \big(\mathop-(n-1)(n+4) f_1 f_3 + (2n-1)(2n+4) f_0 f_4 \big) f^n_{2n-1} \\
  & \phantom{=}\ - 2n(n+4) f_0 f_3 f^n_{2n}.
\end{align*}

We have expressions for $f^n_{2n-5}, \ldots, f^n_{2n+2}$ in terms of $f^n_{2n-3}, \ldots, f^n_{2n}$, and we obtain $v_{n+1}$ via
\begin{align*}
 f^{n+1}_{2n-1} & = f_4 f^n_{2n-5} + \cdots + f_0 f^n_{2n-1}, \\
& \  \, \vdots \\
 f^{n+1}_{2n+2} & = f_4 f^n_{2n-2} + \cdots + f_0 f^n_{2n+2}.
\end{align*}
After some algebraic manipulation we obtain the matrix
\[
M_n = \left[
\begin{array}{rrrr}
   (-(n+3) f_3^2 + 4(n+2)f_2 f_4) J_1    & f_3 J_2     & 4 f_4 J_3   & (2n+3) f_1 f_4 J_4 \\
   (-2 f_2 f_3 + 6(n+2) f_1 f_4) J_1     & 2 f_2 J_2   & 3 f_3 J_3   & (4(2n+1) f_0 f_4 + (n+2) f_1 f_3) J_4 \\
   ((n-1) f_1 f_3 + 8(n+2) f_0 f_4) J_1  & 3 f_1 J_2   & 2 f_2 J_3   & (3(2n+1) f_0 f_3 + f_1 f_2) J_4 \\
       2n f_0 f_3 J_1                    & 4 f_0 J_2   & f_1 J_3     & (2(2n+1) f_0 f_2 - n f_1^2) J_4
\end{array}
\right],
\]
where
\begin{align*}
 J_1 & = (n+1)(2n+1)f_0, \\
 J_2 & = (n+1)(2n+1)(2n+5) f_0 f_4, \\
 J_3 & = 2(n+1)(n+2)(2n+5) f_0 f_4, \\
 J_4 & = (n+2)(2n+5) f_4,
\end{align*}
and the denominator
 \[ D_n = 2(n+2)(2n+1)(2n+5) f_0 f_4. \]
 
Recall that $v_n = v_0 M_0 \cdots M_{n-1} / (D_0 \cdots D_{n-1})$. For each admissible prime $p = 2n+1 \geq 5$, the $p$-adic valuation of $D_0 \cdots D_{n-1}$ is exactly $1$, since $p$ divides $D_{n-2} = 2n(2n-3)(2n+1) f_0 f_4$ exactly once, and $p$ does not divide $D_j$ for $j = n-1$ or any $0 \leq j \leq n - 3$. We may thus compute $v_n \bmod p$ as
 \[ \left(\frac{v_0 M_0 \cdots M_{n-1} \bmod p^2}{D_0 \cdots D_{n-1} \bmod p^2}\right)\bmod p. \]

With additional care it is possible to perform the bulk of the computation working modulo~$p$ rather than $p^2$. As noted above, $D_0 \cdots D_{n-3}$ is a $p$-adic unit, and a direct calculation shows that the entries of the last column of $M_{n-2} M_{n-1}$ are divisible by $2n + 1$. Let $U_n$ be the last column of $M_{n-2} M_{n-1} / D_{n-2} D_{n-1}$. Then $U_n$ is $p$-integral for $p = 2n + 1$, and we may compute the last entry of $v_n \bmod p$, i.e., the lone entry of the Hasse--Witt matrix $W_p$, as
 \[ \left(\left(\frac{v_0 M_0 \cdots M_{n-3} \bmod p}{D_0 \cdots D_{n-3} \bmod p}\right) U_n \right) \bmod p. \]

\begin{remark}
Returning briefly to the general case, we can now see why it always suffices to work with moduli $m_n = p^{d+1}$, for sufficiently large admissible $p$. The denominator $D_n$ always has the form $D_n = C f_0^\alpha f_d^\beta \prod_{i=1}^e (a_i n + b_i)$, where $C, a_i, b_i \in \Z$ and $\alpha$, $\beta$ and $e$ are non-negative integers with $\alpha + \beta \leq d$ and $e \leq d$. We may assume that $(a_i, b_i) = 1$ for all $i$. If $p$ is larger than every prime divisor of $a_i$, we see that $a_i n + b_i$ is divisible by $p$ if and only if $n = -b_i / a_i \bmod p$, and this occurs for at most one value of $n$ in the interval $0 \leq n < (p-1)/2$. Moreover for large enough $p$ we see that $a_i n + b_i$ cannot be divisible by $p^2$ for such $n$. Thus for all sufficiently large admissible primes $p = 2n + 1$, we find that $D_0 \cdots D_{n-1}$ has $p$-adic valuation at most $d$.
\end{remark}

\begin{remark}
One can make $f_3=0$ by replacing $x$ with $x-f_3/(4f_4)$ and $y$ with $y/(16f_4^2)$ and then clearing denominators. This has the advantage that a factor of $f_4$ cancels in the above formulae for $M_n$ and $D_n$, but it will also tend to increase the size of the other coefficients. In general, one can always make $f_{d-1}=0$ with a similar substitution, and when $d$ is even this allows us to remove a power of $f_d$ from $D_n$ and the entries of $M_n$.
\end{remark}

When $f_0 = 0$ one can follow the procedure above, using \eqref{eq:right2} and \eqref{eq:left2} in place of \eqref{eq:right} and \eqref{eq:left}; alternatively, one may switch to a cubic model via the substitution $x = 1/u$, $y = v/u^2$, which is discussed in the next section. Both methods lead to essentially the same formulae.

\subsection{Genus $1$, cubic model}
We now consider the case $g=1$ with $f(x) = f_3 x^3 + f_2 x^2 + f_1 x + f_0$ and $d = 3$.
Assuming $f_0\ne 0$, we obtain the $3\times 3$ transition matrix
\[
M_n = 
\left[
\begin{array}{rrr}
2(n+1)(2n+1)f_0f_2 & 6(n+1)(n+3)f_0f_3 & (n+3)(n+2)f_1f_3\\
4(n+1)(2n+1)f_0f_1 & 4(n+1)(n+3)f_0f_2 & (n+3)(3(2n+1)f_0f_3+f_1f_2)\\
6(n+1)(2n+1)f_0^2 & 2(n+1)(n+3)f_0f_1 & (n+3)(2(2n+1)f_0f_2-nf_1^2)
\end{array}
\right]
\]
with denominator
\[ D_n = 2(n+3)(2n+1)f_0. \]
For all admissible primes $p=2n+1 \geq 5$, the partial product $D_0 \cdots D_{n-1}$ is prime to $p$.

\begin{remark}
In the cubic case one can make $f_3=1$ and $f_2=0$ with a suitable substitution; this simplifies the formulae but may increase the size of $f_0$ and $f_1$.  If the cubic $f(x)$ has a rational root, one can make $f_0=0$ by translating the root to zero (in which case $f_2$ will typically be nonzero).  This is usually well worth doing, since it reduces the dimension of $M_n$ from 3 to 2 (see below).
Similar remarks apply whenever $d$ is odd.
\end{remark}

When $f_0=0$ we have $y^2 = f_3 x^3 + f_2 x^2 + f_1 x$ and the $2\times 2$ transition matrix
\[
M_n =
\left[
\begin{array}{rr}
(n+1)f_2 & 2(n+2)f_3\\
2(n+1)f_1 & (n+2)f_2
\end{array}
\right]
\]
with denominator
\[ D_n = n+2. \]
For all admissible primes $p=2n+1$ the partial product $D_0 \cdots D_{n-1}$ is prime to $p$. 

\subsection{Genus $2$}
The computations in genus 2 are similar, except now each Hasse--Witt matrix has two rows, which we obtain by computing $v_n^{(i)}$ for $i=1,2$.
For the sake of brevity, we omit the details and list only the denominators $D^{(i)}_n$; a Sage \cite{sage-6.0} script for generating the transition matrices $M^{(i)}_n$ is available at \cite{HS:HasseWittWorksheet}.

For $i=1$ we get the denominators
\[
D_n^{(1)} =\begin{cases}
8(n+2)(2n+1)(2n+3)(4n+7)(4n+9)f_0f_6^3&\text{if $d=6,f_0\ne 0$},\\
6(n+2)(2n+1)(3n+5)(3n+7)f_0f_5^2&\text{if $d=5,f_0\ne 0$},\\
3(n+2)(3n+4)(3n+5)f_5^2&\text{if $d=5,f_0=0$}.
\end{cases}
\]
In the case $d=6$, one verifies that the last two columns of $M^{(1)}_{n-1} / D^{(1)}_{n-1}$ are $p$-integral for $p = 2n+1$, and that $D^{(1)}_0 \cdots D^{(1)}_{n-2}$ is a $p$-adic unit except possibly for a single factor of $p$ contributed by $4m+7$ when $m = (n-3)/2$ or by $4m+9$ when $m = (n-4)/2$ (at most one of these occurs for each $p$). Thus the desired row of the Hasse--Witt matrix $W_p$ may be computed as the last two entries of
 \[ \left(\left(\frac{v_0 M^{(1)}_0 \cdots M^{(1)}_{n-2} \bmod p^2}{D^{(1)}_0 \cdots D^{(1)}_{n-2} \bmod p^2}\right) \frac{M^{(1)}_{n-1}}{D^{(1)}_{n-1}}\right)\bmod p. \]
Similar observations apply to both of the $d = 5$ cases, and again one finds that it suffices to work with the moduli $m_n = p^2$ (we omit the details).

The denominators for $i=2$ are
\[
D_n^{(2)} =\begin{cases}
8(n+3)(2n+1)(2n+5)(4n+3)(4n+5)f_0^3f_6&\text{if $d=6,f_0\ne 0$},\\
8(n+4)(2n+1)(4n+3)(4n+5)f_0^3&\text{if $d=5,f_0\ne 0$},\\
3(n+3)(3n+2)(3n+4)f_1^2&\text{if $d=5,f_0=0$}.
\end{cases}
\]
As above, in all three cases one can arrange to use the moduli $m_n=p^2$.

\subsection{Genus $3$}

For $i=1$ we get the denominators
\[
D^{(1)}_n = \begin{cases}
72(n+2)(2n+1)(2n+3)(3n+4)(3n+5)(6n+11)(6n+13)f_0f_8^5&\text{if $d=8,f_0\ne 0$},\\
10(n+2)(2n+1)(5n+7)(5n+8)(5n+9)(5n+11)f_0f_7^4&\text{if $d=7,f_0\ne 0$},\\
5(n+2)(5n+6)(5n+7)(5n+8)(5n+9)f_7^4&\text{if $d=7,f_0=0$}.
\end{cases}
\]
For $i=2$ the denominators are
\[
D^{(2)}_n = \begin{cases}
8(n+2)(2n+1)(2n+5)(4n+3)(4n+5)(4n+7)(4n+9)f_0^3f_8^3&\text{if $d=8,f_0\ne 0$},\\
24(n+2)(2n+1)(3n+7)(3n+8)(4n+3)(4n+5)f_0^3f_7^2&\text{if $d=7,f_0\ne 0$},\\
3(n+2)(3n+2)(3n+4)(3n+5)(3n+7)f_1^2f_7^2&\text{if $d=7,f_0=0$},
\end{cases}
\]
and for $i = 3$ they are
\[
D^{(3)}_n = \begin{cases}
72(n+3)(2n+1)(2n+7)(3n+2)(3n+4)(6n+5)(6n+7)f_0^5f_8&\text{if $d=8,f_0\ne 0$},\\
72(n+5)(2n+1)(3n+2)(3n+4)(6n+5)(6n+7)f_0^5&\text{if $d=7,f_0\ne 0$},\\
5(n+4)(5n+3)(5n+4)(5n+6)(5n+7)f_1^4&\text{if $d=7,f_0=0$}.
\end{cases}
\]
In all three cases it is not difficult to show that by pulling out at most the last three factors from $D_0 \cdots D_{n-1}$, it suffices to compute the partial products modulo $m_n=p^3$, where $p=2n+1$.

\vspace{-6pt}
\section{Accumulating remainder trees}\label{sec:ART}

Given a sequence of $r\times r$ integer matrices $A_0,\ldots,A_{b-2}$, an $r$-dimensional integer row vector~$V$, and a sequence of positive integer moduli $m_1,\ldots,m_{b-1}$, we wish to compute the sequence of reduced row vectors $C_1,\ldots,C_{b-1}$, where
\[
C_n := V A_0 \cdots A_{n-1} \bmod m_n.
\]
For convenience, we define $m_0=1$, so $C_0$ is the zero vector, and we let $A_{b-1}$ be the identity matrix.
We also make the simplifying assumption that the bound $b=2^\ell$ is a power of two, although this is not necessary.
In terms of the prime bound $N$ of the previous sections, we use $b=N/2$, which can be viewed as a bound on $n=(p-1)/2$.

As in \cite[\S 3]{Harvey:HyperellipticPolytime}, we work with complete binary trees of depth $\ell$ with nodes indexed by pairs $(i,j)$ with $0\le i\le \ell$ and $0\le j < 2^i$. For each node we define
\begin{align}
m_{i,j} &:= m_{j 2^{\ell-i}} m_{j 2^{\ell-i} + 1} \cdots m_{(j+1)2^{\ell-i}-1}, \notag \\
A_{i,j} &:= A_{j 2^{\ell-i}} A_{j 2^{\ell-i} + 1} \cdots A_{(j+1)2^{\ell-i}-1}, \label{eq:definetrees} \\
C_{i,j} &:= V A_{i,0} \cdots A_{i,j-1}  \bmod m_{i,j}. \notag
\end{align}
The values $m_{i,j}$ and $A_{i,j}$ may be viewed as nodes in a \emph{product tree}, in which each node is the product of its children,
with leaves $m_j=m_{\ell,j}$ and $A_j=A_{\ell,j}$, for $0 \leq j < b$.
Each vector $C_{i,j}$ is the product of $V$ and all the matrices $A_{i,k}$ that are nodes on the same level and to the left of $A_{i,j}$, reduced modulo $m_{i,j}$.
To compute the vectors $C_j=C_{\ell,j}$, we use the following algorithm.
\bigskip

\noindent
\textbf{Algorithm} \textsc{RemainderTree}
\vspace{2pt}

\noindent
Given $V, A_0,\ldots,A_{b-1}$ and $m_0,\ldots,m_{b-1}$, with $b=2^\ell$, compute $m_{i,j}, A_{i,j}$, and $C_{i,j}$ as follows:
\begin{enumerate}[1.]
\item Set $m_{\ell,j}=m_j$ and $A_{\ell,j}=A_j$, for $0\le j < b$.
\item For $i$ from $\ell-1$ down to 1:\\
\phantom{For} For $0\le j < 2^i$, set $m_{i,j}=m_{i+1,2j}m_{i+1,2j+1}$ and $A_{i,j}=A_{i+1,2j}A_{i+1,2j+1}$.
\item Set $C_{0,0}=V \bmod m_{0,0}$ and then for $i$ from 1 to $\ell$:\\
\phantom{Set }For $0 \leq j < 2^i$ set $C_{i,j} =
\begin{cases}
C_{i-1,\lfloor j/2\rfloor}\bmod m_{i,j}\qquad&\text{if $j$ is even,}\\ 
C_{i-1,\lfloor j/2\rfloor}A_{i,j-1}\bmod m_{i,j}&\text{if $j$ is odd.}\\ 
\end{cases}$
\end{enumerate}
\bigskip

To illustrate the algorithm, let us compute $(p-1)!\bmod p$ for the odd primes $p < 15$; this does not correspond to the computation of a Hasse--Witt matrix, but this makes no difference as far as the \textsc{RemainderTree} algorithm is concerned.
We use odd moduli $m_n=2n+1$ for $0\le n < 8$, except that we set the composite moduli $m_4$ and $m_7$ to 1, and we use $1\times 1$ matrices $A_n = [(2n+1)(2n+2)]$ for $0\le n < 7$, and let $A_7=[1]$ and $V = [1]$.
The trees $m_{i,j}$, $A_{i,j}$, and $C_{i,j}$ computed by the \textsc{RemainderTree} algorithm are depicted below.
\smallskip
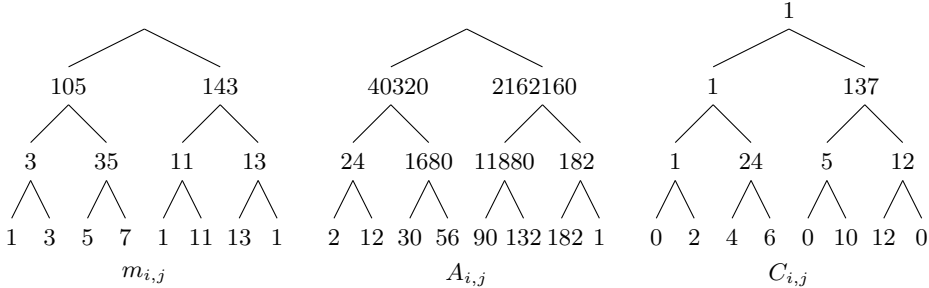
\begin{center}
\begin{tikzpicture}[scale=0.5]
\small
\node at (-8.5,6) {};
\node at (0,6) {};
\node at (8.5,6) {1};

\node at (-10.5,4) {105};
\node at (-6.5,4) {143};
\node at (-1.8,4) {40320};
\node at (1.8,4) {2162160};
\node at (6.5,4) {1};
\node at (10.4,4) {137};

\node at (-11.5,2) {3};
\node at (-9.5,2) {35};
\node at (-7.5,2) {11};
\node at (-5.6,2) {13};
\node at (-3,2) {24};
\node at (-1,2) {1680};
\node at (1,2) {11880};
\node at (2.9,2) {182};
\node at (5.5,2) {1};
\node at (7.5,2) {24};
\node at (9.5,2) {5};
\node at (11.5,2) {12};

\node at (-12,0) {1};
\node at (-11,0) {3};
\node at (-10,0) {5};
\node at (-9,0) {7};
\node at (-8,0) {1};
\node at (-7,0) {11};
\node at (-6,0) {13};
\node at (-5,0) {1};

\node at (-3.5,0) {2};
\node at (-2.5,0) {12};
\node at (-1.5,0) {30};
\node at (-0.5,0) {56};
\node at (0.5,0) {90};
\node at (1.5,0) {132};
\node at (2.6,0) {182};
\node at (3.5,0) {1};
\node at (5,0) {0};
\node at (6,0) {2};
\node at (7,0) {4};
\node at (8,0) {6};
\node at (9,0) {0};
\node at (10,0) {10};
\node at (11,0) {12};
\node at (12,0) {0};

\draw (-8.5,5.5) -- (-10.5,4.5);
\draw (-8.5,5.5) -- (-6.5,4.5);
\draw (0,5.5) -- (-2,4.5);
\draw (0,5.5) -- (2,4.5);
\draw (8.5,5.5) -- (10.5,4.5);
\draw (8.5,5.5) -- (6.5,4.5);

\draw (-10.5,3.5) -- (-11.5,2.5);
\draw (-10.5,3.5) -- (-9.5,2.5);
\draw (-6.5,3.5) -- (-7.5,2.5);
\draw (-6.5,3.5) -- (-5.5,2.5);
\draw (-2,3.5) -- (-3,2.5);
\draw (-2,3.5) -- (-1,2.5);
\draw (2,3.5) -- (1,2.5);
\draw (2,3.5) -- (3,2.5);
\draw (6.5,3.5) -- (5.5,2.5);
\draw (6.5,3.5) -- (7.5,2.5);
\draw (10.5,3.5) -- (9.5,2.5);
\draw (10.5,3.5) -- (11.5,2.5);

\draw (-11.5,1.5) -- (-12,0.5);
\draw (-11.5,1.5) -- (-11,0.5);
\draw (-9.5,1.5) -- (-10,0.5);
\draw (-9.5,1.5) -- (-9,0.5);
\draw (-7.5,1.5) -- (-8,0.5);
\draw (-7.5,1.5) -- (-7,0.5);
\draw (-5.5,1.5) -- (-6,0.5);
\draw (-5.5,1.5) -- (-5,0.5);
\draw (-3,1.5) -- (-3.5,0.5);
\draw (-3,1.5) -- (-2.5,0.5);
\draw (-1,1.5) -- (-1.5,0.5);
\draw (-1,1.5) -- (-0.5,0.5);
\draw (1,1.5) -- (0.5,0.5);
\draw (1,1.5) -- (1.5,0.5);
\draw (3,1.5) -- (2.5,0.5);
\draw (3,1.5) -- (3.5,0.5);
\draw (11.5,1.5) -- (12,0.5);
\draw (11.5,1.5) -- (11,0.5);
\draw (9.5,1.5) -- (10,0.5);
\draw (9.5,1.5) -- (9,0.5);
\draw (7.5,1.5) -- (8,0.5);
\draw (7.5,1.5) -- (7,0.5);
\draw (5.5,1.5) -- (6,0.5);
\draw (5.5,1.5) -- (5,0.5);

\node at (-8.5,-1) {$m_{i,j}$};
\node at (0,-1) {$A_{i,j}$};
\node at (8.5,-1) {$C_{i,j}$};
\end{tikzpicture}
\vspace{-2pt}
\end{center}

\begin{theorem}\label{thm:RTcomplexity}
Let $B$ be an upper bound on the bit-size of $\prod_{j=0}^{b-1}m_j$, let $B'$ be an upper bound on the bit-size of any entry of $V$,
let $h$ be an upper bound on the bit-size of any $m_0,\ldots,m_{b-1}$ and any entry in $A_0,\ldots,A_{b-1}$, and assume that $\log r = O(h)$.
The running time of the \textsc{RemainderTree} algorithm is
\[
O(r^3 \M(B+bh)\log b + r\M(B')),
\]
and its space complexity is
\[
O(r^2 (B+bh) \log b + rB').
\]
\end{theorem}
\begin{proof}
There are $O(B)$ bits at each level of the $m_{i,j}$ tree. For the $A_{i,j}$ tree, observe that the entries of any product $A_{j_1} \cdots A_{j_2-1}$ have bit-size $O((j_2-j_1)h + \log r)$; thus there are $O(bh)$ bits at each level of the $A_{i,j}$ tree. These estimates account for the main terms in the time and space bounds; for more details see the proofs of \cite[Thm.\ 1.1]{Harvey:WilsonPrimes} or  \cite[Prop.\ 4]{Harvey:HyperellipticPolytime}.
We assume classical matrix multiplication throughout, with complexity $O(r^3)$.
The terms involving $B'$ cover any additional cost due to the initial reduction of $V$ modulo $m_{0,0}$.
\end{proof}

\subsection{A fast space-efficient remainder tree algorithm}
The algorithm given in the previous section uses more space than is necessary.
We now describe a more space-efficient approach that is also faster by a significant constant factor.
As above, we assume $b=2^\ell$ is a power of two.
Our strategy is to pick a parameter $k$, and rather than computing a single remainder tree, separately compute the $2^k$ subtrees corresponding to the bottom $\ell - k$ layers of the original tree, each of which has height $\ell-k$ and $t = 2^{\ell-k}$ leaves.

For $0 \leq s < 2^k$, we define the $s$th tree as follows. Let
\begin{align*}
 m^s_j  &:= m_{st+j} \quad (0 \leq j < t), \\
 A^s_j  &:= A_{st+j} \quad (0 \leq j < t), \\
 V^s    &:= V A_0 \cdots A_{st-1} \bmod m_{st} \cdots m_{b-1}.
\end{align*}
For $0 \leq i \leq \ell - k$ and $0 \leq j < 2^i$ we define $m^s_{i,j}$, $A^s_{i,j}$ and $C^s_{i,j}$ in terms of the above data, in direct analogy with \eqref{eq:definetrees}.

We then have $m^s_{i,j} = m_{i+k, j + 2^i s}$ and $A^s_{i,j} = A_{i+k, j + 2^i s}$; in other words, the $m^s_{i,j}$ and $A^s_{i,j}$ trees are identical to the corresponding subtrees of the original $m_{i,j}$ and $A_{i,j}$ trees rooted at the node $(k,s)$. The same is true for the $C^s_{i,j}$ tree, namely, we have $C^s_{i,j} = C_{i+k, j + 2^i s}$. To see this, observe that
 \[ V^s = V A_{k,0} \cdots A_{k,s-1} \bmod m_{k,s} \cdots m_{k,2^k-1}, \]
and $m^s_{0,0} = m_{k,s}$. Therefore
 \[ C^s_{0,0} = V^s \bmod m^s_{0,0} = V A_{k,0} \cdots A_{k,s-1} \bmod m_{k,s} = C_{k,s}. \]
For the the remaining nodes, the claim $C^s_{i,j} = C_{i+k, j + 2^i s}$ follows by working downwards from the root of the $C^s$ tree.

The idea of the \textsc{RemainderForest} algorithm below is to compute each subtree separately, allowing us to reuse space, and to keep track of the vector $V^s$ and the moduli product
 \[ Y^s := m_{st} \cdots m_{b-1} \]
as we proceed from one subtree to the next. The \textsc{RemainderTree} algorithm may be viewed as a special case of the \textsc{RemainderForest} algorithm, using $k = 0$.
\bigskip

\noindent
\textbf{Algorithm} \textsc{RemainderForest}
\vspace{2pt}

\noindent
Given $V,A_0,\ldots,A_{b-1}$ and $m_0,\ldots,m_{b-1}$, with $b=2^\ell$, and an integer $k\in [0,\ell]$, compute $C_0,\ldots,C_{b-1}$ as follows:
\begin{enumerate}[1.]
\item Set $Y^0 \leftarrow m_0 \cdots m_{b-1}$ and $V^0 \leftarrow V \bmod Y^0$, and let $t=2^{\ell-k}$.
\item For $s$ from $0$ to $2^k-1$:
\begin{enumerate}[a.]
\item Call \textsc{RemainderTree} with inputs $V^s$, $A_{st}, \ldots, A_{(s+1)t-1}$, and $m_{st}, \ldots, m_{(s+1)t-1}$\\
to compute trees $m^s$, $A^s$, $C^s$.
\item Set $Y^{s+1} \leftarrow Y^s/m^s_{0,0}$ and $V^{s+1} \leftarrow V^s A^s_{0,0}\bmod Y^{s+1}$.
\item Output the values $C_{st+j}=C^s_{e,j}$ for $0\le j < t$. 
\item Discard $Y^s$, $V^s$, and the trees $m^s$, $A^s$, $C^s$.
\end{enumerate}
\end{enumerate}
\bigskip

\noindent
We now bound the complexity of the \textsc{RemainderForest} algorithm.
We do not include the size of the input in our space bound; in the context of computing Hasse--Witt matrices the input matrices $A_j$ are dynamically computed as they are needed, in blocks of size $2^{\ell-k}$.

\begin{theorem}\label{thm:RFcomplexity}
Let $B$ be an upper bound on the bit-size of $\,\prod_{j=0}^{b-1}m_j$ such that $B/2^k$ is an upper bound on the bit-size of $\prod_{j=st}^{st+t-1}m_j$ for all $s$.
Let $B'$ be an upper bound on the bit-size of any entry of $V$,
let~$h$ be an upper bound on the bit-size of any $m_0,\ldots,m_{b-1}$ and any entry in $A_0,\ldots,A_{b-1}$, and assume that $\log r = O(h)$.
The running time of the \textsc{RemainderForest} algorithm is
\[
O(r^3\M(B+bh)(\ell-k) + 2^kr^2\M(B)+r\M(B')),
\]
and its space complexity is
\[
O(2^{-k}r^2(B+bh)(\ell-k)+r(B+B')).
\]
\end{theorem}
\begin{proof}
The time complexity of step 1 is $O(\M(B) \log b + r\M(B+B'))$.
There are $2^k$ calls to \textsc{RemainderTree} in step 2, each of which takes time
\[
O(r^3\M(2^{-k}B+2^{-k}bh)(\ell-k)+r\M(B)),
\]
by Theorem~\ref{thm:RTcomplexity}, since the bit-size of any entry of any $V^{s}$ is bounded by $O(B)$.
The cost of step~2b is  bounded by $O(\M(B) + r^2\M(B + 2^{-k} bh))$, thus each invocation of step 2 costs
\[
O(r^3 \M(2^{-k}B + 2^{-k} bh)(\ell - k) + r^2 \M(B)).
\]
Multiplying by $2^k$ yields the desired time bound.
The first term in the space bound matches the corresponding term in Theorem~\ref{thm:RTcomplexity}; the second term bounds the space needed for step 1 (and the output), and dominates the second term in the space bound of Theorem~\ref{thm:RTcomplexity}.
\end{proof}

With $k=0$ we have $\ell-k=\ell=\log_2 b$, and the bounds in Theorem \ref{thm:RFcomplexity} reduce to those of Theorem \ref{thm:RTcomplexity}.
With $k=\ell$ the \textsc{RemainderForest} algorithm has essentially optimal space complexity $O(rbh)$ (matching the size of its output), but its time complexity is then quasi-quadratic in $b$, rather than quasi-linear. The intermediate choice $k = \log_2 \ell + O(1)$ yields a time complexity that is at least as good as that of the \textsc{RemainderTree} algorithm (and may be  smaller by a significant constant factor), but with the space complexity improved by a factor of $\log b$. We will see below that for computing Hasse--Witt matrices, $bh$ is somewhat larger than~$B$, and this implies that an even better choice is $k = 2\log_2 \ell + O(1)$, reducing the space complexity by a further factor of $\log b$.
See Table~\ref{table:HWtimingsk} in \S \ref{sec:details} for an explicit example.

\begin{remark}
The space complexity can be further reduced using a time-space trade-off as described in \cite[Thm.\ 1.2]{Harvey:WilsonPrimes}.
In practice we find that when computing Hasse--Witt matrices using the \textsc{RemainderForest} approach, for $g\le 3$ and the range of $N$ of interest to us, space is not a limiting factor and no time-space trade-off is necessary.  See \S \ref{sec:details} for further details.
\end{remark}

\subsection{Computing the Hasse--Witt matrix}

We now give a complete algorithm for computing the Hasse--Witt matrix $W_p$ of a hyperelliptic curve at all admissible primes $p\le N$; as noted above, the bound $N$ on $p$ corresponds to a bound of $b=N/2$ on $n$.
While the basic approach has been explained in the previous sections, to achieve the best space complexity we must interleave the \textsc{RemainderForest} computations involving the matrices~$M_n$ and denominators $D_n$, so we use \textsc{RemainderTree} to directly handle each subtree, rather than using \textsc{RemainderForest} as a black box.
This also allows us to more carefully control the size of the moduli that we use, as discussed further below.

\bigskip

\noindent
\textbf{Algorithm} \textsc{ComputeHasseWittMatrices}
\vspace{2pt}

\noindent
Given a hyperelliptic curve $C\colon y^2 = f(x)=\sum_{i=0}^{d}f_ix^i$ of genus $g$, compute the Hasse--Witt matrices $W_p$ for admissible primes $p\le N$ as follows:
\begin{enumerate}[1.]
\item Construct a list $\mathcal{P}$ of the admissible primes $p=2n+1\le N$.
\item For $i$ from 1 to $g$:
\begin{enumerate}[a.]
\item Compute $M^{(i)}\in\Z[n]^{r\times r}$ and $D^{(i)}\in\Z[n]$ satisfying $v_{n+1}^{(i)} = v_n^{(i)}M^{(i)}(n)/D^{(i)}(n)$, as in \S\ref{sec:RR}.
\item Use \textsc{ComputeHasseWittRows} below to compute the $i$th row of $W_p$ for all $p\in\mathcal{P}$.
\end{enumerate}
\item Output the matrices $W_p$.
\end{enumerate}
\bigskip 

As discussed in \S\ref{sec:RR}, in order to minimize the power of $p=2n+1$ that we use as our moduli, let $e$ and $w$ be integers such that $p^e$ does not divide $D_0\cdots D_{n-1-w}$ for all sufficiently large admissible $p$. For $g\le 3$ using $e=g$ and $w\le 3$ suffices; in general $e$ and $w$ are both $O(g)$.  Our strategy is to compute the partial products $M_0\cdots M_{n-1-w}$ and $D_0\cdots D_{n-1-w}$ modulo $p^e$ using remainder trees, and to handle the last $w$ values of $M_j$ and $D_j$ separately;
this allows us to use a smaller value of $e$ than would otherwise be possible.
In the context of the \textsc{RemainderTree} algorithm, this means shifting the moduli $m_j$ by $w$ places to the left, relative to the $A_j$.
\bigskip

\noindent
\textbf{Algorithm} \textsc{ComputeHasseWittRows}
\vspace{2pt}

\noindent
Given $i \in [1, g]$, positive integers $e,w$, a list $\mathcal{P}$ of admissible primes $p \le N=2^{\ell+1}$, a matrix $M^{(i)}\in\Z[n]^{r\times r}$, and $D^{(i)}\in \Z[n]$,  compute the $i$th row of $W_p$ for all $p\in\mathcal{P}$ as follows:
\begin{enumerate}[1.]
\item Compute $Y=\prod_{p\in\mathcal{P}} p^g$, let $v=1$, and let $V\in \Z^r$ be the $(r-i+1)$th standard basis vector.
\item Fix $k=2\log_2 (\ell \sqrt g)+O(1)$, let $t=2^{\ell-k}$, and for $s$ from 0 to $2^k-1$:
\begin{enumerate}[a.]
\item For $st\le j < (s+1)t$, set $m_j=p^e=(2j+1+2w)^e$ if $p\in\mathcal{P}$ and $1$ otherwise.
\item Compute $M_j=M(j)$ and $D_j=D(j)$ for $st \le j < (s+1)t+w-1$.
\item Call \textsc{RemainderTree} with inputs $V,M_j,m_j$ to compute $C_j=V\prod_{u=0}^{j-1}M_u\bmod m_j$, $m^s = \prod m_j$, and $M^s=\prod M_j$, where $j$ ranges over integers from $st$ to $st+t-1$.
\item Call \textsc{RemainderTree} with inputs $v,D_j,m_j$ to compute $c_j=v\prod_{u=0}^{j-1}D_u\bmod m_j$ and $D^s=\prod D_j$, where $j$ ranges over integers from $st$ to $st+t-1$.
\item Set $Y\leftarrow Y/m^s$, $V\leftarrow V M^s\bmod Y$, and $v\leftarrow v D^s\bmod Y$.
\item Compute $v_j = C_jM_j\cdots M_{j+w-1}/(c_jD_j\cdots D_{j+w-1}) \bmod p$ for $st\le j < (s+1)t$ such that $p=2j+1+2w\in\mathcal{P}$, and extract the $i$th row of $W_p$ as the last $g$ entries of $v_j$.
\end{enumerate}
\item Output the $i$th row of each of the matrices $W_p$ for $p\in\mathcal{P}$.
\end{enumerate}
\bigskip

We now prove the main result announced in the introduction, which bounds the time and space complexity of \textsc{ComputeHasseWittMatrices} by $O(g^5\M(N\log (\norm f N))\log N)$ and $O(g^2 N(1+\log\norm f/\log N))$, respectively, assuming $g=O(\log N)$.

\begin{proof}[of Theorem \ref{thm:main}]
The time and space needed to enumerate the primes in $[1,N]$ may be bounded by $O(N\log^{2+\epsilon}N)$ and $O(N)$, respectively, via \cite[Prop.~2.3]{Harvey:WilsonPrimes}, by dividing the interval $[1,N]$ into $O(\log^3 N)$ subintervals.
It follows from Chebyshev's bound that $\mathcal{P}$ uses $O(N)$ space.
The complexity of \textsc{ComputeHasseWittRows} may be bounded as in the proof of Theorem~\ref{thm:RFcomplexity};
the only new elements are steps 2a and 2b, which have a total time complexity of $O(g^4 N\M(\log (\norm f N)))$, and step 2f, whose complexity is lower.  This is within our desired time bound, and the space complexity of these steps is dominated by the size of the output.

We now proceed as in the proof of Theorem~\ref{thm:RFcomplexity}.
We have $B=O(gN$), since $\sum_{p\le N}\log p \sim N$, and we note that the requirement that $B/2^k$ bound the bit-size of the product $m_{st}\cdots m_{st+t-1}$ is satisfied for any $k=O(\log \log N)$; these facts follow from the prime number theorem.
Further, $b=N/2$, $B'=O(1)$, $\ell=\log_2N-1$, $k=2\log_2(\ell \sqrt g)+O(1)$, and $h=O(g\log(\norm f N))$, since the polynomials in $M(n)$ and $D(n)$ all have degree $O(g)$ and coefficients of bit-size $O(g\log \norm f)$, and the moduli have bit-size $O(g\log N)$.
This yields the time bound
\[
O(g^3\M(gN+hN)\log N +g^3 \M(gN)\log^2 N  + g),
\]
and the space bound
\[
O(g^2(gN+hN)\log N/(g \log^2 N) + g^2N).
\]
This yields $O(g^4\M(N\log (\norm f N))\log N)$ time and $O(g^2 N(1+\log \norm f/\log N))$ space bounds for \textsc{ComputeHasseWittRows}, which is called $g$ times.
\end{proof}

\section{Implementation details and performance results}\label{sec:details}

We implemented the \textsc{ComputeHasseWittMatrices} algorithm in C, using the \texttt{gcc} compiler~\cite{gcc} and the GNU multiple-precision arithmetic library (GMP) \cite{gmp}. For the crucial operation of multiplying matrices with very large integer entries, we used a customized FFT implementation as described below.

\vspace{-6pt}
\subsection{Customized FFT}

The customized FFT uses the standard ``small primes'' approach, as outlined in \cite[Ch.~8]{vzGG-compalg}. To compute a product $uv$, where $u, v \in \Z$, we choose a parameter $c \geq 1$ and write $u = F(2^c)$ and $v = G(2^c)$, where $F,G \in \Z[x]$ have coefficients bounded by $2^c$. We then compute the polynomial product $FG \in \Z[x]$ and obtain $uv$ as $(FG)(2^c)$. To compute $FG$, we choose four suitable 62-bit primes $p_1, \ldots, p_4$ and compute $FG \bmod p_i$ in $(\Z/p_i \Z)[x]$ for each $i$, and then reconstruct $FG$ via the Chinese remainder theorem. The parameter $c$ is chosen as large as possible so that the coefficients of $FG$ remain bounded by $p_1 \cdots p_4$. Multiplication in $(\Z/p_i \Z)[x]$ is achieved by using Fourier transforms (number-theoretic transforms) over $\Z/p_i\Z$. This requires $p_i = 1 \bmod 2^{a}$, where $2^{a}$ is the transform length. Our implementation uses optimized modular arithmetic as in~\cite{Har-NTT}, truncated Fourier transforms to avoid power-of-two jumps in running times \cite{vdH-TFT-apps,vdH-2}, and ideas from \cite{Har-cachetft} to improve locality.

To multiply matrices we use the same strategy. If $u$ and $v$ are $r \times r$ integer matrices (recall that $r=d$ or $d-1$, where $d$ is the degree of the polynomial $f$ in the curve equation $y^2=f(x)$), we write $u = F(2^c)$ and $v = G(2^c)$ where now $F$ and $G$ are matrices of polynomials with small coefficients, or equivalently polynomials with matrix coefficients. We then perform $2r^2$ forward transforms, multiply the resulting Fourier coefficients (each coefficient is an $r \times r$ matrix over $\Z/p_i \Z)$, and perform $r^2$ inverse transforms, with a final linear-time substitution generating the desired product $uv$. Our implementation allows the polynomial entries to have signed coefficients, so that we can directly handle matrices $u$ and $v$ containing a mixture of positive and negative entries. Matrix-vector products are handled similarly.

The main advantage of this approach over a straightforward GMP implementation is that we require only $O(r^2)$ transforms rather than $O(r^3)$. In our computations the Fourier transforms make up the bulk of the time spent on matrix multiplication.

\vspace{-6pt}
\subsection{Timings}

The timings listed in this section were obtained using an 8-core Intel Xeon E5-2670 CPU running at 2.60GHz, with 20MB of cache and 32GB of RAM; in each case we list the total CPU time, in seconds, for a single-threaded implementation.
Table~\ref{table:HWtimings} lists timings for increasing values of $N$ with $g=1,2,3$ and each of the three possible values of $r$;
as in \S \ref{sec:RR} we have
\[
r = 
\begin{cases}
2g&\text{when $d=2g+1$ and $f_0=0$},\\
2g+1&\text{when $d=2g+1$ and $f_0\ne 0$},\\
2g+2&\text{when $d=2g+2$ and $f_0\ne 0$}.
\end{cases}
\]
Table~\ref{table:HWspace} gives the corresponding memory consumption for each case.

\begin{table}
\normalsize
\setlength{\tabcolsep}{6pt}
\begin{center}
\begin{tabular}{@{}lrrrrrrrrrrrr@{}}
&\multicolumn{12}{c}{$k$}\\
\cmidrule{2-13}
&0&1&2&3&4&5&6&7&8&9&10&11\\
\midrule
time (s)&750&718&661&602&535&483&\textbf{459}&466&540&736&1145&2055\\
space (MB) &8529&4416&2215&1089&533&311&\textbf{220}&178&162&153&149&147\\
\bottomrule
\vspace{2pt}
\end{tabular}
\caption{Time (CPU seconds) and space (MB) for Hasse--Witt matrix computations for the curve $y^2 = 2x^7+3x^6+5x^5+7x^4+11x^3+13x^2\
+17x+19$ with $N=20$ and varying $k$.}\label{table:HWtimingsk}
\end{center}
\end{table}

\begin{table}
\normalsize
\setlength{\tabcolsep}{9pt}
\begin{center}
\begin{tabular}{@{}rrrrrrrrrr@{}}
&\multicolumn{3}{c}{$g=1$}&\multicolumn{3}{c}{$g=2$}&\multicolumn{3}{c}{$g=3$}\\
\cmidrule(r){2-4}\cmidrule(r){5-7}\cmidrule(r){8-10}
$N$&$r=2$&$r=3$&$r=4$&$r=4$&$r=5$&$r=6$&$r=6$&$r=7$&$r=8$\\
\midrule
$2^{14}$&$<1$&$<1$&$<1$&$<1$&$<1$&1&1&2&3\\
$2^{15}$&$<1$&$<1$&$<1$&1&1&2&3&6&9\\
$2^{16}$&$<1$&$<1$&1&2&3&5&8&14&21\\
$2^{17}$&$<1$&1&1&4&7&12&20&34&52\\
$2^{18}$&1&2&4&9&17&29&49&81&123\\
$2^{19}$&1&4&8&22&40&69&116&192&294\\
$2^{20}$&3&9&20&50&94&166&282&459&694\\
$2^{21}$&7&21&47&123&227&398&667&1085&1633\\
$2^{22}$&17&49&114&287&534&946&1560&2540&3810\\
$2^{23}$&38&115&268&645&1240&2230&3660&5940&9100\\
$2^{24}$&89&271&641&1510&2920&5260&8490&13800&20600\\
$2^{25}$&202&628&1470&3430&6740&11800&19600&31800&47200\\
$2^{26}$&470&1475&3390&7930&15800&27400&44700&72900&107000\\
\bottomrule
\vspace{2pt}
\end{tabular}
\caption{Time (CPU seconds) for Hasse--Witt matrix computations for the curve $y^2 = 2x^d + 3x^{d-1} + \cdots + p_{d+1}$, where $p_n$ is the $n$th prime ($f_0 = 0$ for $r=2g$).}\label{table:HWtimings}
\end{center}
\end{table}

The impact of varying the parameter $k$, which determines the number $2^k$ of subtrees used in the \textsc{RemainderForest} algorithm, is illustrated for a particular example with $g=3$ and $N=20$ in Table~\ref{table:HWtimingsk}.
In all of our other tests the parameter $k$ was chosen to optimize time; the optimal choice of $k$ varies with both $N$ and $r$ and in our tests ranged from $4$ to $8$.
As can be seen in Table~\ref{table:HWtimingsk}, the value of $k$ that optimizes time also yields a space utilization that is much better than would be achieved by the original \textsc{RemainderTree} algorithm (the case $k=0$).
Even in our largest tests, the time-optimal value of $k$ yielded a space utilization under 20GB, well within the 32GB available on our test system.
By contrast, the original \textsc{RemainderTree} algorithm would have required more than 1TB of memory in our larger tests.

\begin{table}
\normalsize
\setlength{\tabcolsep}{9pt}
\begin{center}
\begin{tabular}{@{}rrrrrrrrrr@{}}
&\multicolumn{3}{c}{$g=1$}&\multicolumn{3}{c}{$g=2$}&\multicolumn{3}{c}{$g=3$}\\
\cmidrule(r){2-4}\cmidrule(r){5-7}\cmidrule(r){8-10}
$N$&$r=2$&$r=3$&$r=4$&$r=4$&$r=5$&$r=6$&$r=6$&$r=7$&$r=8$\\
\midrule
$2^{14}$&$<1$&$<1$&$<1$&$<1$&$<1$&$<1$&$<1$&$<1$&1\\
$2^{15}$&$<1$&$<1$&1&1&1&1&4&6&8\\
$2^{16}$&1&1&1&3&5&7&9&12&16\\
$2^{17}$&2&2&4&6&10&14&18&25&33\\
$2^{18}$&5&5&8&13&20&29&38&51&69\\
$2^{19}$&11&11&17&27&41&59&79&106&144\\
$2^{20}$&16&21&35&53&83&121&162&220&295\\
$2^{21}$&32&42&71&108&169&249&332&450&610\\
$2^{22}$&63&84&145&218&346&517&682&942&1258\\
$2^{23}$&124&170&307&444&716&1064&1396&1940&2614\\
$2^{24}$&247&634&634&920&1467&2195&2869&3980&5385\\
$2^{25}$&498&708&1300&1890&3014&3398&5865&8231&11162\\
$2^{26}$&1002&1440&2679&3843&6478&6950&12134&12925&17137\\
\bottomrule
\vspace{2pt}
\end{tabular}
\caption{Space (MB) for Hasse--Witt matrix computations for  the curve $y^2 = 2x^d + 3x^{d-1} + \cdots + p_{d+1}$, where $p_n$ is the $n$th prime ($f_0=0$ for $r=2g$).}\label{table:HWspace}
\end{center}
\end{table}


Tables~\ref{table:smalljac} compares the performance of the new algorithm (in the column labelled \texttt{hassewitt}) to the \texttt{smalljac} implementation described in \cite{KS:Lseries}.
In genus 2 the \texttt{smalljac} implementation relies primarily on group computations in the Jacobian of the curve, as described in \cite{KS:Lseries}, and the current version \cite{smalljac} includes additional improvements from \cite{Sutherland:p-groups}.
As can be seen in the table, the new algorithm surpasses the performance of \texttt{smalljac} when $N$ is between $2^{18}$ and~$2^{19}$ and is more than 12 times faster for $N=2^{26}$.

As noted in \cite{KS:Lseries}, for genus 3 curves, Harvey's optimization \cite{Har-kedlaya} of Kedlaya's algorithm~\cite{Kedlaya:algorithm} is faster than using group computations in the Jacobian for $N\ge 2^{16}$.
Table~\ref{table:hypellfrob} compares the performance of the new algorithm to an implementation based on Harvey's \texttt{hypellfrob} library~\cite{hypellfrob}, using one digit of $p$-adic precision (sufficient to compute the Hasse--Witt matrix).
In genus 3 the new algorithm is substantially faster than \texttt{hypellfrob} for all the values of $N$ that we tested, and more than 20 times faster for $N=2^{26}$.
We did not include a column for the case $r=8$ in Table~\ref{table:hypellfrob} because the \texttt{hypellfrob} library requires~$d$ to be odd.
\begin{table}
\normalsize
\setlength{\tabcolsep}{8pt}
\begin{center}
\begin{tabular}{@{}rrrrrrr@{}}
&\multicolumn{2}{c}{$r=4$}&\multicolumn{2}{c}{$r=5$}&\multicolumn{2}{c}{$r=6$}\\
\cmidrule(r){2-3}\cmidrule(r){4-5}\cmidrule(r){6-7}
$N$&\texttt{hassewitt}&\texttt{smalljac}&\texttt{hassewitt}&\texttt{smalljac}&\texttt{hassewitt}&\texttt{smalljac}\\
\midrule
$2^{14}$&0.2&0.2&0.4&0.2&0.7&0.3\\
$2^{15}$&0.6&0.5&1.1&0.6&1.9&0.7\\
$2^{16}$&1.4&1.7&2.8&1.7&4.9&2.0\\
$2^{17}$&3.5&5.6&6.8&5.6&11.9&6.4\\
$2^{18}$&8.6&19.9&16.8&20.2&29.0&22.1\\
$2^{19}$&20.6&76.0&39.7&76.4&69.1&83.4\\
$2^{20}$&48.9&257\phantom{.0}&94.4&257\phantom{.0}&166\phantom{.0}&284\phantom{.0}\\
$2^{21}$&123\phantom{.0}&828\phantom{.0}&227\phantom{.0}&828\phantom{.0}&398\phantom{.0}&914\phantom{.0}\\
$2^{22}$&287\phantom{.0}&2630\phantom{.0}&534\phantom{.0}&2630\phantom{.0}&946\phantom{.0}&2900\phantom{.0}\\
$2^{23}$&645\phantom{.0}&8560\phantom{.0}&1240\phantom{.0}&8570\phantom{.0}&2230\phantom{.0}&9520\phantom{.0}\\
$2^{24}$&1510\phantom{.0}&28000\phantom{.0}&2920\phantom{.0}&28000\phantom{.0}&5260\phantom{.0}&31100\phantom{.0}\\
$2^{25}$&3430\phantom{.0}&92200\phantom{.0}&6740\phantom{.0}&92300\phantom{.0}&11800\phantom{.0}&102000\phantom{.0}\\
$2^{26}$&7930\phantom{.0}&314000\phantom{.0}&15800\phantom{.0}&316000\phantom{.0}&27400\phantom{.0}&349000\phantom{.0}\\
\bottomrule
\vspace{2pt}
\end{tabular}
\caption{Performance comparison with \emph{\texttt{smalljac}} in genus $2$. Times in CPU seconds.}\label{table:smalljac}
\end{center}
\end{table}

\begin{table}
\normalsize
\begin{center}
\begin{tabular}{@{}rrrrr@{}}
&\multicolumn{2}{c}{$r=6$}&\multicolumn{2}{c}{$r=7$}\\
\cmidrule(r){2-3}\cmidrule(r){4-5}
$N$&\texttt{hassewitt}&\texttt{hypellfrob}&\texttt{hassewitt}&\texttt{hypellfrob}\\
\midrule
$2^{14}$&1.3&6.7&2.0&6.8\\
$2^{15}$&3.4&15.5&5.5&15.6\\
$2^{16}$&8.3&37.4&13.6&37.6\\
$2^{17}$&20.2&95.1&33.3&95.0\\
$2^{18}$&48.6&249\phantom{.0}&80.4&250\phantom{.0}\\
$2^{19}$&116\phantom{.0}&680\phantom{.0}&192\phantom{.0}&681\phantom{.0}\\
$2^{20}$&282\phantom{.0}&1910\phantom{.0}&459\phantom{.0}&1920\phantom{.0}\\
$2^{21}$&667\phantom{.0}&5450\phantom{.0}&1090\phantom{.0}&5460\phantom{.0}\\
$2^{22}$&1560\phantom{.0}&16200\phantom{.0}&2540\phantom{.0}&16300\phantom{.0}\\
$2^{23}$&3660\phantom{.0}&49400\phantom{.0}&5940\phantom{.0}&49400\phantom{.0}\\
$2^{24}$&8490\phantom{.0}&152000\phantom{.0}&13800\phantom{.0}&152000\phantom{.0}\\
$2^{25}$&19600\phantom{.0}&467000\phantom{.0}&31800\phantom{.0}&467000\phantom{.0}\\
$2^{26}$&44700\phantom{.0}&1490000\phantom{.0}&72900\phantom{.0}&1490000\phantom{.0}\\
\bottomrule
\vspace{2pt}
\end{tabular}
\caption{Performance comparison with \emph{\texttt{hypellfrob}} in genus $3$. Times in CPU seconds.}\label{table:hypellfrob}
\end{center}
\end{table}

\bibliographystyle{amsplain}
\providecommand{\bysame}{\leavevmode\hbox to3em{\hrulefill}\thinspace}
\providecommand{\MR}{\relax\ifhmode\unskip\space\fi MR }
\providecommand{\MRhref}[2]{%
  \href{http://www.ams.org/mathscinet-getitem?mr=#1}{#2}
}
\providecommand{\href}[2]{#2}

\affiliationone{David Harvey\\
School of Mathematics and Statistics\\
University of New South Wales\\
Sydney NSW 2052\\
Australia
\email{d.harvey@unsw.edu.au}
}
\affiliationtwo{Andrew V. Sutherland\\
Department of Mathematics\\
Massachusetts Institute of Technology\\
Cambridge, MA  02139\\
USA
\email{drew@math.mit.edu}
}

\end{document}